\documentclass[accepted]{uai2026} 
                        
\usepackage[american]{babel}
\usepackage{amsmath}

\usepackage{newtxtext}
\usepackage[subscriptcorrection]{newtxmath}
\usepackage{natbib} 
\usepackage{mathtools} 
\usepackage{booktabs} 
\usepackage{tikz} 

\usepackage{graphicx}
\usepackage{caption}

\usepackage{tikz}
\usepackage{mathtools}
\usetikzlibrary{arrows.meta, decorations.pathmorphing, positioning, decorations.markings}

\newcommand{\bR}{\mathbb{R}}
\newcommand{\bN}{\mathbb{N}}

\newcommand{\cB}{\mathcal{B}}

\newcommand{\sfE}{{\sf E}}

\newcommand{\sfP}{{\sf P}}

\newcommand{\sfBF}{{\sf BF}}

\newcommand{\var}{\text{var}}
\newcommand{\cov}{\text{cov}}

\newcommand{\pa}{{\rm pa}}
\newcommand{\nd}{{\rm nd}}

\newcommand{\rt}{\right}
\newcommand{\lt}{\left}

\def\T{{ \mathrm{\scriptscriptstyle T} }}

\usepackage{amsthm}

\newtheorem{theorem}{Theorem}
\newtheorem{lemma}{Lemma}

\theoremstyle{definition}

\newtheorem{example}{Example}

\title{Consistent Bayesian causal discovery for\\ structural equation models with equal error variances
}

\author[1]{\href{mailto:<anamitra.chaudhuri@austin.utexas.edu>?Subject=Your UAI 2026 paper}{Anamitra~Chaudhuri}}
\author[1]{Yang~Ni}
\author[2]{Anirban~Bhattacharya}
\affil[1]{%
    Department of Statistics and Data Sciences\\
    The University of Texas at Austin\\
    Austin, Texas, USA
}
\affil[2]{%
    Department of Statistics\\
    Texas A\&M University\\
    College Station, Texas, USA
}
  
\begin{document}
\maketitle

\begin{abstract}
We consider the problem of recovering the true causal structure among a set of variables, generated by a linear acyclic structural equation model (SEM) with the error terms being independent, not necessarily Gaussian, and having equal variances. It is well-known that the true underlying directed acyclic graph (DAG) encoding the causal structure is uniquely identifiable under this assumption. Interestingly, in this setting, it further holds that  
the sum of minimum expected squared errors for every variable, while predicted by the best linear combination of its parent variables, is minimised if and only if the causal structure is represented by any supergraph of the true  DAG.
In this work, we propose a Bayesian DAG selection method, where the working model assumes Gaussian SEM with equal error variances, and employ independent $g$-priors on each set of SEM coefficients. Furthermore, we utilise the aforementioned key property to establish that the proposed method recovers the true graph consistently without any additional distributional assumption, and illustrate it with a simulation study.
\end{abstract}

\section{Introduction}

The field of causal discovery aims to learn the presence and direction of causal relationships, often from purely observational data, which enables the prediction of intervention outcomes when controlled experimentation is infeasible. This is critical in various scientific fields such as 
public health \citep{shen2020challenges}, 
neuroscience \citep{zhou2023functional}, climate science \citep{runge2019inferring}, 
psychology \citep{ni2025causal}, 
philosophy \citep{glymour2019review}, economics \citep{imbens2004nonparametric}, and 
to recent domains of machine learning and artificial intelligence, including causal representation learning \citep{scholkopf2021toward, zhang2024causal}, and causal transfer learning \citep{zhang2017transfer}. 

This paper considers the problem of learning causal structures from purely observational data within the framework of causal Bayesian networks, represented by directed acyclic graphs (DAGs) \citep{pearl2009causality}. In general, DAGs are identifiable only up to their Markov equivalence class, in which all DAGs encode the same conditional independencies \citep{heckerman1995learning}. 
Numerous methods have been proposed to 
estimate the Markov equivalence class, such as the Peter–Clark (PC) algorithm
\citep{spirtes2001causation}, and the Greedy Equivalence Search (GES) algorithm \citep{chickering2002optimal}; see \cite{drton2017structure} for a review. 
An alternative prominent line of work along this direction over the past two decades is Bayesian structure learning, which leverages Markov chain Monte Carlo (MCMC) methods to traverse the model space and facilitates posterior inference on relevant quantities via model averaging; see, for instance, \cite{madigan1996bayesian, geiger2002parameter}.
Bayesian methods are particularly appealing in this setting because they provide principled and coherent uncertainty quantification for both structural features and downstream quantities while explicitly accounting for model uncertainty, and at the same time offer substantial flexibility in incorporating prior information \citep{friedman2003being, castelletti2018learning}. For example, prior knowledge in the form of a reference network, such as a known biological pathway, can be encoded to encourage shrinkage toward scientifically plausible structures while still allowing the data to update and refine the model \citep{werhli2007reconstructing}.


Notably, a series of recent work has demonstrated that the exact DAG, rather than its Markov equivalence class, can be uniquely identified from observational data under {\it additional} distributional assumptions. For example, if the causal relationships are represented by some structural equation model \citep{bollen1989structural}, then unique recovery of the DAG is possible when the structural equation model (SEM) is linear with all errors being non-Gaussian \citep{shimizu2006linear}.
Curiously, if the errors have {\it equal variance}, Gaussian or not, then exact identification is again possible \citep{peters2014identifiability, chen2019causal} -- this setting is the primary focus of this paper.

Despite its growing popularity, Bayesian DAG structure learning has traditionally focused on developing efficient computational algorithms for Gaussian models \citep{giudici2003improving,
niinimaki2011partial, goudie2016gibbs, kuipers2017partition}, with far fewer contributions addressing general non-Gaussian cases \citep{hoyer2009bayesian, shimizu2014bayesian}. While extensive simulations in these works demonstrate, in general, superior performance over non-Bayesian methods across a wide range of non-Gaussian distributions, theoretically principled Bayesian selection frameworks with guarantees such as consistency are still lacking, especially under the equal error variance assumption, due to challenges including the asymptotic analysis of Bayes factors under model misspecification, and we address this research gap in this work.

Specifically, under this equal-variance assumption, 
it has been established \citep{loh2014high, aragam2015learning}
that the {\it sum} of the {\it minimum expected squared errors} from linearly regressing each variable on its parents 
is minimized by any supergraph of the true data-generating DAG. We provide an alternative proof technique, where the key to establishing this result, namely Theorem \ref{thm:key_result}, lies in a regression formulation for the diagonal entries of the Cholesky factorization of a covariance matrix \citep{pourahmadi2007cholesky}.  
Theorem \ref{thm:key_result} has important implications which we carefully employ towards proposing our Bayesian structure learning method. Precisely, under a working Gaussian structural equation model with equal error variances, 
and assuming independent $g$-priors on each set of SEM coefficients, the marginal likelihood for each DAG involves an empirical version of the sum of least-squared errors. Consequently, the key result is utilized to establish posterior DAG selection consistency in Theorem \ref{thm:post_const}, contributing to a growing body of literature \citep{cao2019posterior, lee2019minimax,zhou2023complexity,chaudhuri2025consistent} on this topic. The results are further illustrated with simulation studies.

\paragraph*{Notations} We write $\bR$ for the set of real numbers and $\bN := {1,2,\dots}$ for that of the natural numbers, and for any $n \in \bN$, let $[n] := {1,2,\dots,n}$. A DAG is denoted by a pair $\gamma=(V,E)$ with $V=[p]$ the set of $p$ nodes and $E \subset V \times V$ the set of directed edges such that for $k, j \in V$, if there is a directed edge from node $k$ to node $j$, 
then $(k,j) \in E$, 
in which case we call node $k$ a \emph{parent} of node $j$ in $\gamma$, and the set of its parents is subsequently denoted by $\pa^\gamma(j)$. Moreover, the total number of edges in $\gamma$ is represented by $|\gamma|$, and thus, $|\gamma| = \sum_{j = 1}^p |\pa^\gamma(j)|$. 
The collection of all DAGs with $p$ nodes is denoted by $\Gamma^p$.
For $\gamma' \in \Gamma^p$ with edge set $E'$, we 
write $\gamma' \supseteq \gamma$, with a slight abuse of notation,  if $E' \supseteq E$, i.e., every directed edge in $\gamma$ is present in $\gamma'$, or in other words, $\gamma'$ is a supergraph of $\gamma$. Finally, for any $x = (x_1, \ldots, x_p)^\T \in \bR^p$, and $I \subset [p]$, we denote by $x_I$ the subvector of $x$ consisting of the elements $x_k, k \in I$.

\section{Structural causal model}

We consider $p$ random variables $X_j$, $j \in [p]$, and assume that they are generated by a linear, recursive SEM associated with a data-generating true DAG $\gamma^*\in \Gamma^p$ with nodes $[p]$ corresponding to the random variables and edges $E^*$ representing their direct causal relationships: for $j,k \in [p]$, we have $(k, j) \in E^*$ when $X_k$ has a \emph{direct linear (causal) effect} on $X_j$. Acyclicity guarantees that there exists a permutation $c^*(\cdot)$ of $[p]$, which we call the \emph{causal order} of the variables, such that, for every $j \in [p]$, if the causal order of $X_j$ is $c^*(j)$, then $(k, j) \in E^*$ only if $c^*(k) < c^*(j)$. Equivalently, each node’s parents precede it in the causal order.
For every $j\in [p]$, we let $\pa^*(j) \equiv \pa^{\gamma^*}(j)$ be the parent set of node $j$ in $\gamma^*$. 
Then the SEM posits that $X_j$ is some (unknown) linear function of $X_{\pa^*(j)}$ with an additive (unobserved) independent error $\epsilon_j$,
\begin{equation}\label{eq:model}
\qquad X_j = X_{\pa^*(j)}^\T \beta^*_j + \epsilon_j,\quad \epsilon_j \;\; {\buildrel\rm ind\over\sim} \;\; \sfP^*_j, \qquad j \in [p],
\end{equation}
where the elements in the (unknown) SEM coefficient vector $\beta^*_{j} \in \bR^{|\pa^*(j)|}$ are {\it non-zero} and quantify the direct causal effects of $X_k, k \in \pa^*(j)$, on $X_j$. Regarding the distributions of the errors $\sfP^*_j, j \in [p]$, we only assume that $E(\epsilon_j) = 0$, and $\var(\epsilon_j) = \sigma^2$ for every $j \in [p]$, i.e., the equal error variance assumption \citep{peters2014identifiability, chen2019causal}.
Moreover, due to the independence of the errors, $\sfP^*$, the joint probability distribution of the errors, admits the form $\sfP^* = \otimes_{j=1}^p \sfP^*_j$, which in turn induces the joint probability distribution of $X = (X_1, \ldots, X_p)^\T$ through \eqref{eq:model}.
We consider $n$ independent and identically distributed (iid) observations of $X$, denoted by $X^{(i)} = (X_{1}^{(i)}, \ldots, X_{p}^{(i)})^\T,
\ i \in [n]$, and denote the complete dataset by $D_n := \{X^{(i)} : i \in [n]\}$ .
We illustrate the above in a concrete example below.
\begin{example}\label{ex:4node}
Consider $p = 4$ with $\gamma^*$ being the DAG as shown in Figure \ref{fig:4node}, 
 \begin{figure}[h]
\centering
\begin{tikzpicture}[->, >=stealth, thick]
\node (3) [circle, draw, minimum size=1cm] at (1,1.732) {3}; 
\node (2) [circle, draw, minimum size=1cm] at (0,0) {2}; 
\node (1) [circle, draw, minimum size=1cm] at (2,0) {1}; 
\node (4) [circle, draw, minimum size=1cm] at (4,0) {4};
\draw (3) -- (2);
\draw (3) -- (1);
\draw (1) -- (4);
\draw (3) -- (4);
\end{tikzpicture}        
\caption{DAG $\gamma^*$}
\label{fig:4node}
\end{figure}
\noindent and let the associated data-generating SEM in \eqref{eq:model} take the following specific form: 
\begin{align*}
X_3 &= \epsilon_3,\\
X_2 &= 1.8 X_3 + \epsilon_2,\\
X_1 &= -2.7 X_3 + \epsilon_1,\\
X_4 &= 0.9 X_3 + 3.6 X_1 + \epsilon_4,
\end{align*} 
where the errors $\epsilon_j, j \in [4]$ are independent with mean $0$ and variance $\sigma^2 = 1.6$.
\end{example}

 
Interestingly, under the aforementioned model, if we minimise over $\Gamma^p$ the sum of nodewise minimum expected squared errors, obtained while predicting each variable with the best linear function of its parents, then the minimum is attained with any supergraph of $\gamma^*$, while the summands being equal to the common error variance $\sigma^2$ \citep{loh2014high, aragam2015learning}, as presented below. We denote by $\sfE_*[\cdot]$ the expectation under $\sfP^*$.

\begin{theorem}\label{thm:key_result}
For every $\gamma \in \Gamma^p$, let $r^\gamma := \sum_{j = 1}^p r_j^\gamma$, where
\[
r_j^\gamma := \min_{\beta_j} \ \sfE_*(X_j - X_{\pa^\gamma(j)}^\T\beta_j)^2, \quad j \in [p].
\]
In particular, when $\gamma = \gamma^*$, we denote the above quantities by $r^*_j, j \in [p]$, and let $r^* := \sum_{j = 1}^p r^*_j$. 
    Then we have $r^\gamma \geq r^*$,
    where the equality holds if and only if $\gamma \supseteq \gamma^*$.
\end{theorem}
We provide an alternative proof strategy to derive this result, see Appendix. A critical step of this approach involves, for every $\gamma \in \Gamma^p$, bounding each summand by the minimum expected squared errors when for every variable the best linear prediction is based on all variables with lower causal order under $\gamma$. Notably, the latter quantities coincide with the squared diagonal entries in the Cholesky factor of the covariance matrix of the variables permuted under the causal order of $\gamma$, as shown in \cite{pourahmadi2007cholesky}, and this fact is carefully utilised to achieve the desired minimisation.
The above result suggests that the nodewise aggregate of the mean-squared errors, obtained from the least square regressions of every variable upon its parents, is expected to be minimized by the supergraphs of the true DAG. 
Specifically, for every $j \in [p]$, let $X_{j,n} \in \bR^n$ denote the vector of $n$ observations of $X_j$, and $D_{j, n}^\gamma \in \bR^{n \times |\pa^\gamma(j)|}$ denote the data matrix with its rows corresponding to the $n$ observations 
of $X_{\pa^\gamma(j)}$,
and for every $\gamma \in \Gamma^p$, let $R_n^\gamma := \sum_{j = 1}^p R_{j,n}^\gamma$, where
\begin{align*}
&R_{j,n}^\gamma := n^{-1}X_{j, n}^\T (I_n - P_{j, n}^\gamma) X_{j,n}, \quad \text{and}\\
&P_{j, n}^\gamma := D_{j,n}^\gamma(D_{j,n}^{\gamma \, \T}D_{j,n}^\gamma)^{-1}D_{j,n}^{\gamma \, \T}, \quad \; j \in [p].
\end{align*}
Then, as a consequence of Theorem \ref{thm:key_result}, it is natural to consider $R_n^\gamma$, which asymptotically equates to $r^\gamma$, as a statistic for graph learning, and employ a model complexity penalty that penalizes the number of edges for scoring graphs. For example, one may consider a Bayesian information criterion (BIC)-type scoring criterion, $n R_n^\gamma + |\gamma| \log n$, which, upon minimising over $\Gamma^p$, potentially leads to $\gamma^*$. Interestingly, in a Bayesian context, if we consider a Gaussian SEM with equal error variances as our {\it working model} and apply $g$-priors on the SEM coefficients, then $R_n^\gamma$ appears inside the Bayes factors, and as a result, with a suitable choice of $g$ they become arbitrarily large in favor of $\gamma^*$ against other DAGs, eventually resulting in the posterior DAG selection consistency even if the true data-generating errors may not be Gaussian, as we formally illustrate in the subsequent sections.

\section{Proposed method}

For a fully Bayesian inference of model \eqref{eq:model}, one would have to specify the error distributions $\sfP^\gamma_j$ for each candidate DAG $\gamma$. We show from an asymptotic perspective that it is safe to simply use Gaussian distributions, which leads to straightforward posterior calculation due to the existence of simple conjugate priors.
Specifically, for any DAG $\gamma \in \Gamma^p$, we consider that the observations $X^{(i)}, \ i \in [n]$ are iid and follow the \textit{Gaussian-error} SEM 
with real SEM coefficient vectors $b_j^\gamma \in \bR^{|\pa^\gamma(j)|}, j \in [p]$, 
and positive variance $\theta^\gamma$, 
\begin{align}\label{eq:work_model}
    X_j = X_{\pa^\gamma(j)}^\T b_j^\gamma
+ e_j^\gamma, \quad e_j^\gamma \;\; {\buildrel\rm ind\over\sim} \;\; \text{N} (0, \theta^\gamma), \quad j \in [p].
\end{align}
We treat the above as our \textit{working model} and emphasize here that the true data-generating errors can be any distribution with mean zero and finite variances.
We impose a DAG-g-prior on the SEM coefficients and the non-informative Jeffreys prior on the error variance: 
\begin{align}\label{eq:prior_dist}
\begin{split}
&b_{j}^\gamma \, \big| \, \theta^\gamma, D_{j,n}^\gamma \; \overset{\rm ind}{\sim}  \; \pi_{b, \theta, j}^\gamma(\cdot)= \text{N} (0, g\theta^\gamma(D_{j, n}^{\gamma \,\T} D_{j, n}^\gamma)^{-1}), \\
&\qquad \qquad\theta^\gamma \; \sim \;\, \pi_\theta(\cdot)\propto {1}/{\theta^\gamma}.
\end{split}
\end{align}
Let $b^\gamma := \{b_j^\gamma : j \in [p]\}$, and denote the likelihood function of data by $\ell\lt(D_n\rt | b^\gamma, \theta^\gamma, \gamma)$,
which, upon marginalising over $b^\gamma$ and $\theta^\gamma$, leads us to the marginal likelihood or evidence for DAG $\gamma$,
\begin{align}\notag
m\lt(D_n | \gamma\rt)= \int \ell\lt(D_n\rt | b^\gamma, \theta^\gamma, \gamma) &\Big(\prod_{j =1}^p \pi_{b, \theta, j}^\gamma(b_j^\gamma) \, db_{j}^\gamma\Big) \,\\
\label{eq:marg_m}
&\qquad \times \pi_\theta(\theta^\gamma) \,d\theta^\gamma.
\end{align}
Conveniently, $m\lt(D_n | \gamma\rt)$ admits a closed-form expression under our model-prior combination.


\begin{lemma}\label{lem:marg_gam}
Let $V_n := n^{-1}\sum_{j = 1}^p X_{j,n}^\T X_{j,n}$. 
Then for every $\gamma \in \Gamma^p$, we have \;\;
\[m(D_n | \gamma) \propto 
(V_n + gR_n^\gamma)^{-np/2}
(1 + g)^{(np-|\gamma|)/2}.
\]
\end{lemma}
The proof can be found in the Appendix. Thus, following Lemma \ref{lem:marg_gam}, the Bayes factor in favor of $\gamma$ over any other $\gamma' \in \Gamma^p$, denoted by $\sfBF_n(\gamma, \gamma') := m(D_n | \gamma)/m(D_n | \gamma')$, indeed involves the desired statistics $R_n^\gamma$, as indicated earlier in the previous section.

\section{Consistent DAG selection}


Now, given a DAG prior $\gamma \sim \pi(\cdot)$ on $\Gamma^p$, the posterior probability of $\gamma$ given data $D_n$ is proportional to the product of the marginal likelihood and the DAG prior probability,
\begin{align}\label{eq:post_prob}
\pi(\gamma | D_n) \; \propto \; m\lt(D_n | \gamma\rt) \times \pi(\gamma). 
\end{align} 
The following result establishes the desired \textit{posterior DAG selection consistency}, that is, the posterior probability of the true DAG tends to unity in probability, as sample size grows, under a suitable choice of $g$, and any typical non-informative DAG prior, e.g., the uniform prior $\pi (\cdot) \propto 1$.
\begin{theorem}\label{thm:post_const}
     Suppose that $g = n$ and consider any DAG prior $\pi(\cdot)$ such that there exists $C > 0$ satisfying $\pi(\gamma)/\pi(\gamma') \leq C$ for every $\gamma, \gamma' \in \Gamma^p$.
     Then we have 
     \begin{align*}
         1 - \pi(\gamma^* | D_n) \leq \; &\frac{1}{\surd{n}}\exp(O_p(1))\\
         &+ \exp\lt(-\frac{np}{2}(\delta^* + o_p(1))\rt)(1+n)^{|\gamma^*|/2},
     \end{align*}
     where $\delta^* := \min_{\gamma^* \not\subseteq \gamma} 
     \log r^\gamma - \log (p \sigma^2) > 0$,
     and the $O_p$ and $o_p$ statements are under $\sfP^*$. Moreover, if $\gamma^*$ is a complete graph, then the $n^{-1/2}\exp(O_p(1))$ term in the above is dispensable.
\end{theorem} 
The proof can be found in the Appendix.
The requirement on the prior $\pi(\cdot)$ is minimal, holding for any DAG prior that assigns strictly positive mass over $\Gamma^p$.  
For Gaussian DAGs, \cite{cao2019posterior, lee2019minimax} study selection consistency under the assumption that the true causal order of the variables is {\it known}. \cite{zhou2023complexity} relaxes the latter assumption and with an additional assumption of faithfulness \citep{uhler2013geometry} consistently recovers the Markov equivalence class using a data-dependent prior. In this work, the assumption of Gaussianity is further relaxed, and we establish that even the data-generating DAG can be identified consistently when the associated errors have equal variances, and no other assumption is needed for this purpose.

\section{Simulation study}
In this section, we present a simulation study to illustrate the posterior DAG selection consistency of our proposed method, established in the previous section.
Specifically, we consider $p = 3$, with $\gamma^*$ to be the DAG in Figure \ref{fig:sim_stud1} with the associated data-generating SEM:
\begin{align}\label{eq:sim_stud_SEM1}
\begin{split}
X_1 &= \epsilon_1,\\
X_2 &= 2.7X_1 + \epsilon_2,\\
X_3 &= 1.5X_2 + \epsilon_3,
\end{split}
\end{align}
where the errors are distributed as
$\epsilon_1 \sim \text{Laplace}(0, 1/\sqrt{2})$,
$\epsilon_2 \sim \text{N}(0, 1)$, and
$\epsilon_3 \sim (1/\sqrt{3})\, \text{t}_3$, all having variance $1$.
\begin{figure}[h]
  \centering
        \begin{tikzpicture}[->, >=stealth, thick]
            \node (1) [circle, draw, minimum size=1cm] at (1,1.732) {1};
            \node (2) [circle, draw, minimum size=1cm] at (0,0) {2};
            \node (3) [circle, draw, minimum size=1cm] at (2,0) {3};
            \draw (1) -- (2);
            \draw (2) -- (3);
        \end{tikzpicture}
    \caption{DAG $\gamma^*$.}
    \label{fig:sim_stud1}
\end{figure}

Furthermore, for every $\gamma \in \Gamma^p, j \in [p]$, we consider the prior distributions $\pi_{b, \theta, j}^\gamma(\cdot)$ and $\pi_\theta(\cdot)$ according to \eqref{eq:prior_dist}. Thus, following Theorem \ref{thm:post_const}, employing $g=n$, and the uniform DAG prior $\pi(\cdot) \propto 1$ is sufficient to lead us to the desired posterior DAG selection consistency. 
To illustrate this, or specifically, the asymptotic behavior of the posterior probability of $\gamma^*$, we consider sample sizes $n \in \{100 \times 2^k : k = 4,5,6,7\}$. For each $n$, we generate $100$ independent datasets $D_n$ from the underlying distribution and compute $\pi(\gamma^* | D_n)$
for each replication using the analytically tractable marginal likelihoods from Lemma \ref{lem:marg_gam}.
Indeed, as shown in Figure \ref{fig:boxplot_posterior}, the boxplots of $\pi(\gamma^* | D_n)$shrinks towards $1$, demonstrating the in-probability convergence of $\pi(\gamma^*|D_n)$ as established in Theorem \ref{thm:post_const}.

\begin{figure}[htbp]
    \centering
    \includegraphics[width=0.42\textwidth]{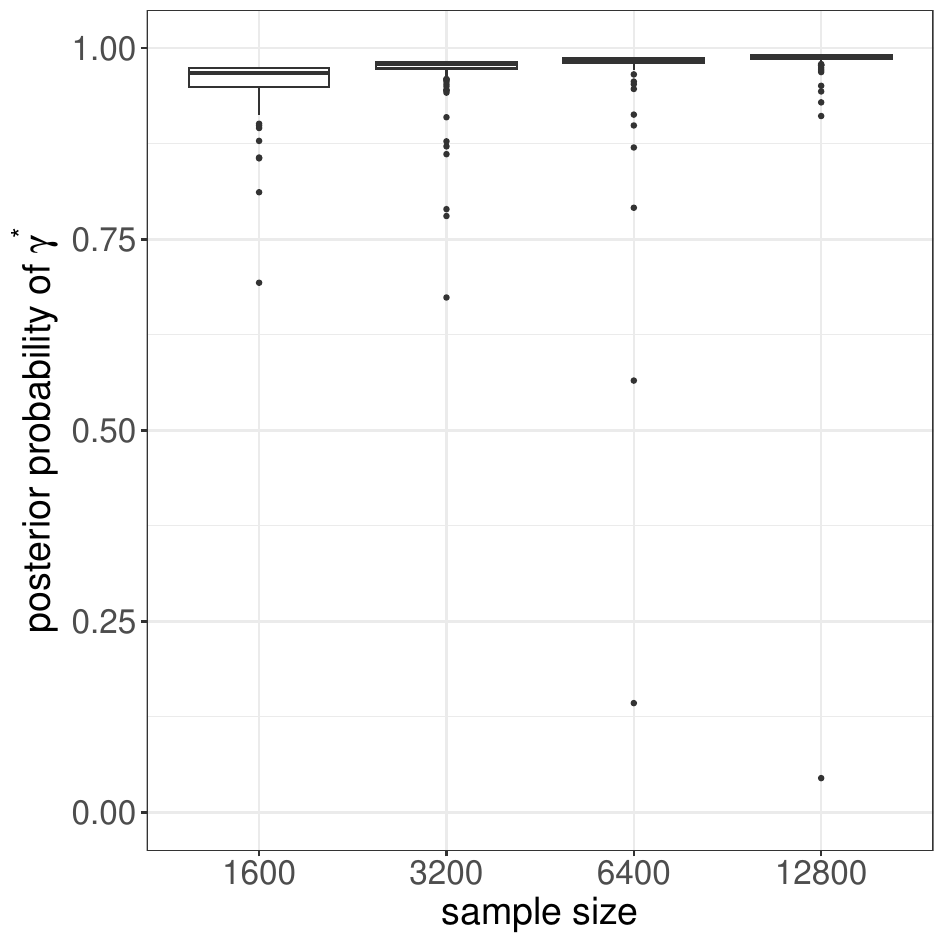}
    \caption{Boxplots of 
    $\pi(\gamma^* \mid D_n)$ over 100 replications for four different sample sizes.}
    \label{fig:boxplot_posterior}
\end{figure}

For additional visualization, see also Figure~\ref{fig:histogram_posterior}, which displays the histogram of $\pi(\gamma^* \mid D_n)$ for $n = 100 \times 2^7$, demonstrating strong posterior concentration near $1$.

\begin{figure}[htbp]
    \centering
    \includegraphics[width=0.42\textwidth]{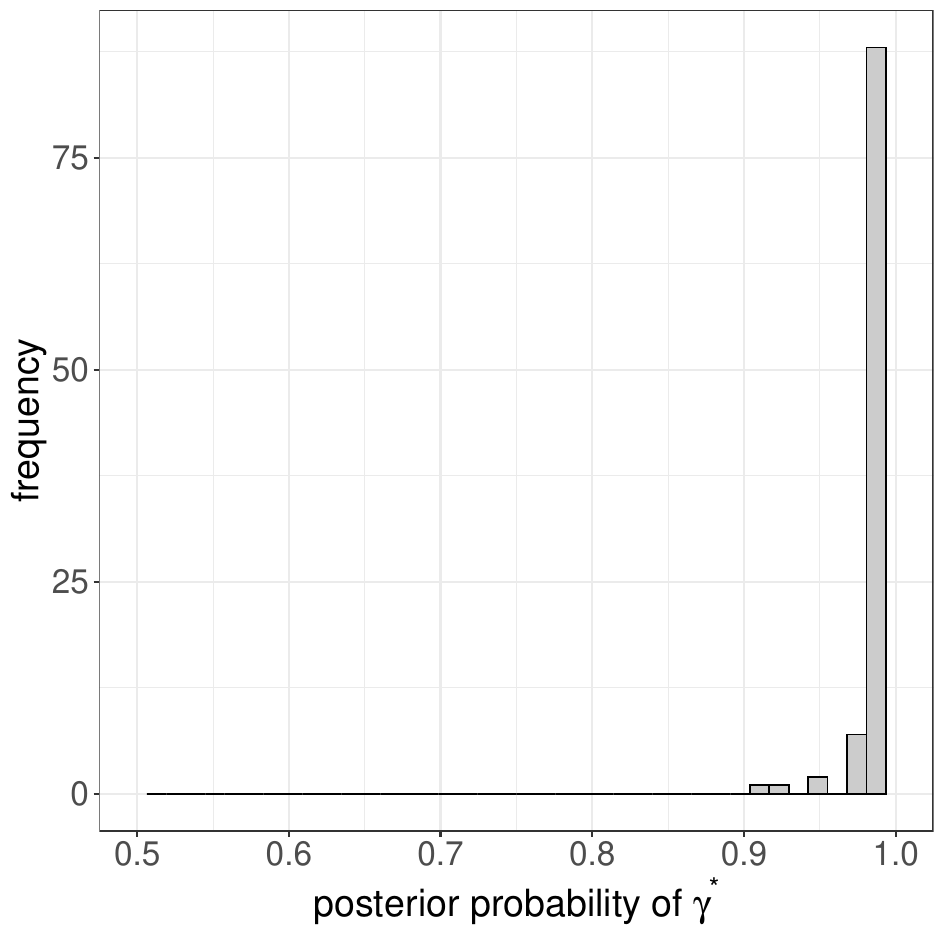}
    \caption{Histogram of 
    $\pi(\gamma^* \mid D_n)$ over 100 replications for sample size $n = 100 \times 2^7$.}
    \label{fig:histogram_posterior}
\end{figure}

\section{Discussion}
In this work, we study Bayesian structure learning for linear recursive SEMs under the equal error variance assumption. To recover the unknown data-generating DAG, we propose a Bayesian SEM with Gaussian errors, assumed to have equal variances, and employ an appropriate $g$-prior to further analyze its theoretical properties under potential model misspecification, with minimal requirement on the DAG prior. We show that the proposed method consistently learns the true DAG, that is, its posterior probability converges to one as the sample size increases, and we simultaneously characterize the corresponding rate of convergence.

Several natural extensions arise from this framework. One direction is to extend the analysis to high-dimensional settings under additional structural assumptions, such as sparsity and suitable tail conditions on the error distributions \citep{johnson2012bayesian}. 
Finally, there are many important open questions for future research in the general context of Bayesian structure learning. One direction is to develop principled and computationally efficient Bayesian DAG selection methods for nonlinear SEMs and establish comparable asymptotic guarantees.
Further extensions include focusing on directed cyclic graphs or non-recursive SEMs, where factorization properties and equivalence characterizations are substantially more complex than in the acyclic case. In addition, the lack of conjugate priors and the resulting intractability of marginal likelihoods complicate theoretical analysis in these broader settings. Finally, incorporating latent confounders or correlated errors presents another important avenue for future work \citep{zhou2022causal, salehkaleybar2020learning, chen2024discovery}.

\begin{acknowledgements} 
    The research of A. Chaudhuri and Y. Ni were supported by NIH R01 GM148974. The research of Y. Ni was additionally supported by NSF DMS-2112943.  The research of A. Bhattacharya was supported partially by NSF DMS-2210689 and NSF DMS-1916371.  

\end{acknowledgements}

\bibliography{BCDeqv}

\newpage

\onecolumn

\title{Consistent Bayesian causal discovery for\\ structural equation models with equal error variances\\(Supplementary Material)}
\maketitle


\appendix

\section{Proof of Theorem~\ref{thm:key_result}}
\begin{proof}
Due to \eqref{eq:model}, 
we have $r^*_j = \var(\epsilon_j) = \sigma^2$ for every $j \in [p]$, which implies $r^* = p\sigma^2$. Without loss of generality, suppose the true causal order corresponds to $(1, \ldots, p)$, i.e., $c^*(j) = j$ for every $j \in [p]$. Then the true model in \eqref{eq:model} can be expressed as $X = \cB^*X + \epsilon$, where $\cB^*$ is a lower triangular matrix with all its diagonal elements being $0$, and $\epsilon = (\epsilon_1, \ldots, \epsilon_p)^\T$. Therefore, since $\cov(\epsilon) = \sigma^2I_p$, and $X = (I_p - \cB^*)^{-1}\epsilon$, we have 
\begin{align}\label{eq:cov_X}
    \cov(X) = \sigma^2(I_p - \cB^*)^{-1}((I_p - \cB^*)^{-1})^\T = LL^\T,
\end{align}
where $L := \sigma(I_p - \cB^*)^{-1}$ is the lower triangular Cholesky factor of $\cov(X)$, and thus, regarding its diagonal elements, we have $L_{jj} = \sigma$ for every $j \in [p]$.

Now, fix $\gamma \in \Gamma^p$. Let the corresponding causal order of the variables be $c(\cdot)$, and for every $j \in [p]$, we denote by $\nd^\gamma(j)$ the set of non-descendants of node $j$ in $\gamma$, defined as any node with lower causal order than node $j$,
i.e., $\nd^\gamma(j) = \{k \in [p] : c(k) < c(j)\}$.
Subsequently, for every $j \in [p]$, $c^{-1}(j)$ corresponds to the variable that has causal order $j$, and there exists a permutation matrix $P$ for which $PX = (X_{c^{-1}(1)}, \ldots, X_{c^{-1}(p)})^\T$. Furthermore, let $\cov(PX) = WW^\T$, where $W$ is the corresponding lower-triangular Cholesky factor, and following that, we have
\begin{align}\label{eq:det_cov_PX}
\begin{split}
\prod_{j = 1}^p W_{jj}^2 &= \det(\cov(PX))\\
&= \det(P) \, \det(\cov(X)) \det(P^\T)
\\
&= \det(\cov(X)) = \prod_{j = 1}^p L_{jj}^2 = \sigma^{2p},
\end{split}
\end{align}
where the fourth equality holds due to \eqref{eq:cov_X}.

Furthermore, regarding the diagonal elements of $W$, we have, for every $j \in [p]$,
\begin{align}\label{ineq:W_r}
    W_{jj}^2 = \min_{\beta_j} \ \sfE_*(X_{c^{-1}(j)} - X_{\nd^\gamma(c^{-1}(j))}^\T\beta_j)^2 \leq r_{c^{-1}(j)}^\gamma,
\end{align}
where the equality follows from \cite{pourahmadi2007cholesky} \S 2.2 as $\nd^\gamma(c^{-1}(j)) = \{c^{-1}(k) : k < j\}$, and the inequality holds due to the fact that $\pa^\gamma(c^{-1}(j))\subseteq \nd^\gamma(c^{-1}(j))$.
Therefore, we have
\begin{align*}
    r^\gamma = \sum_{j = 1}^p r^\gamma_j \geq \sum_{j = 1}^p W_{jj}^2 \geq p \Big(\prod_{j = 1}^p W_{jj}^2\Big)^{1/p} = p\sigma^2 = r^*,
\end{align*}
where the first inequality is due to \eqref{ineq:W_r}, the second one follows from the AM-GM inequality, and the second equality holds due to \eqref{eq:det_cov_PX}. Furthermore, $r^\gamma = r^*$ if and only if equality holds in both of the above inequalities. In the second inequality, equality holds if and only if $W_{jj}^2 = \sigma^2$ for every $j \in [p]$, which is equivalent to having $c(j) = j$ for every $j \in [p]$ due to \eqref{eq:model}. Moreover, in the first inequality, equality holds if and only if equality holds in \eqref{ineq:W_r}, which, as $c^{-1}(j) = j$, is equivalent to having $W_{jj}^2 = r^\gamma_j$ for every $j \in [p]$. This in turn holds if and only if $\pa^*(j) \subseteq \pa^\gamma(j)$ for every $j \in [p]$, or equivalently, $\gamma^* \subseteq \gamma$. The proof is complete.
\end{proof}

\section{Proof of posterior DAG selection consistency}
\begin{proof}[Proof of Lemma \ref{lem:marg_gam}]
Fix $\gamma \in \Gamma^p$ and $j \in [p]$. Then, following \eqref{eq:prior_dist}, we have 
\[D_{j, n}b_j^\gamma | D_{j, n}, \theta^\gamma \; \sim \; \text{N} (0, g\theta^\gamma D_{j,n}^\gamma(D_{j, n}^{\gamma \, \T} D_{j, n}^\gamma)^{-1} D_{j,n}^{\gamma \, \T}) \equiv \text{N}(0, g\theta^\gamma P_{j, n}^\gamma).\] 
Furthermore, due to \eqref{eq:work_model}, we have $X_{j, n} = D_{j, n}^\gamma b_j^\gamma + e_{j, n}^\gamma$, where $e_{j, n}^\gamma \sim \text{N} (0, \theta^\gamma I_n)$. Thus, we have
\[
X_{j,n} | D_{j,n}^\gamma, \theta^\gamma \sim \text{N} (0, \theta^\gamma (gP_{j,n}^\gamma + I_n)),
\] 
which incorporates marginalization over $b_j^\gamma$ in \eqref{eq:marg_m}.
Subsequently, by using standard integral to marginalize over $\theta^\gamma$, we have
\begin{align}
    m(D_n | \gamma) \propto 
    \label{eq:marg_theta}
     \frac{\lt(\sum_{j = 1}^pX_{j,n}^\T(gP_{j,n}^\gamma + I_n)^{-1}X_{j,n}\rt)^{-\frac{np}{2}}}{\prod_{j = 1}^p \det(gP_{j,n}^\gamma + I_n)^{1/2}}.
\end{align}
Now, we use Woodbury matrix identity \citep{henderson1981deriving} to simplify the numerator in \eqref{eq:marg_theta}, 
\begin{align*}
(gP_{j,n}^\gamma + I_n)^{-1} 
&= I_n - gD_{j,n}^\gamma(D_{j,n}^{\gamma \, \T}D_{j,n}^\gamma + gD_{j,n}^{\gamma \, \T}D_{j,n}^\gamma)^{-1}D_{j,n}^{\gamma \, \T}
= \frac{1}{1+g}(I_n + g(I_n - P_{j,n}^\gamma)).
\end{align*}
For the denominator, we apply the generalised matrix determinant lemma, see \cite{harville1997matrix} \S 18.2, 
\begin{align*}
    \det(gP_{j,n}^\gamma + I_n) 
    &= \det(D_{j,n}^{\gamma \, \T}D_{j,n}^\gamma + gD_{j,n}^{\gamma \, \T}D_{j,n}^\gamma) \det((D_{j,n}^{\gamma \, \T}D_{j,n}^\gamma)^{-1})
    = (1 + g)^{|\pa^\gamma(j)|}.
\end{align*}
Substituting the above in \eqref{eq:marg_theta}, 
\begin{align*}
    m(D_n | \gamma) &\propto 
    (1+g)^{\frac{np}{2}}
    \frac{\big(\sum_{j = 1}^pX_{j,n}^\T(I_n + g(I_n - P_{j,n}^\gamma))X_{j,n}\big)^{-\frac{np}{2}}}{\prod_{j = 1}^p (1 + g)^{|\pa^\gamma(j)|/2}}\\
    &= 
    \Big(\sum_{j = 1}^p (X_{j,n}^\T X_{j,n} + gnR_{j,n}^\gamma)\Big)^{-\frac{np}{2}} (1 + g)^{\frac{np}{2} - \frac{1}{2}\sum_{j = 1}^p|\pa^\gamma(j)|}\\
    &= n^{-np/2} (V_n + gR_n^\gamma)^{-np/2}
(1 + g)^{(np -|\gamma|)/2}.
\end{align*}
This completes the proof.
\end{proof}

\begin{lemma}\label{lem:phi_t}
    For any $a, b > 0$, we have $|\log(a+t) - \log(b+t)| \leq |\log a - \log b|$ for every $t \geq 0$.
\end{lemma}
\begin{proof}
    Fix any $a, b > 0$, and let $\phi(t) := \log(a+t) - \log(b+t)$, implying that $\phi'(t) = (b-a)/((a+t)(b+t))$. Thus, $\phi(t)$ is monotone in $t$, and since $\lim_{t \to \infty} \phi(t) = 0$, we have $|\phi(t)| \leq |\phi(0)|$ for every $t \geq 0$, leading to the result.
\end{proof}

In the rest of the paper, we denote $R_n^\gamma$ by $R_n^*$, in particular, when $\gamma = \gamma^*$.
\begin{lemma}\label{lem:mis_wilk_shift}
    If $\gamma^* \subset \gamma$, then $(-np/2)(\log (V_n + nR_n^*) - \log(V_n + nR_n^\gamma)) = O_p(1)$.
\end{lemma}
\begin{proof}
We observe that the log-likelihood ratio test statistic \citep{wilks1938large}, while testing for model selection between the nested working models given by \eqref{eq:work_model} with corresponding DAGs $\gamma^*$ and $\gamma$, admits the form
$(-np/2)(\log R_n^* - \log R_n^\gamma)$. However, since the models are misspecified, we follow \cite{vuong1989likelihood} Theorem 3.3 to derive that it is $O_p(1)$.
    Furthermore, we have
    \begin{align*}
        \lt|-\frac{np}{2}(\log (V_n + nR_n^*) - \log(V_n + nR_n^\gamma))\rt| &=  \lt|-\frac{np}{2}(\log (V_n/n + R_n^*) - \log(V_n/n + R_n^\gamma))\rt|\\
        &\leq \lt|-\frac{np}{2} (\log R_n^* - \log R_n^\gamma)\rt| = O_p(1),
    \end{align*}
    where the inequality follows from Lemma \ref{lem:phi_t} as $V_n > 0$ by definition. 
\end{proof}
\begin{proof}[Proof of Theorem \ref{thm:post_const}]
The 
posterior odds in favor of $\gamma^*$ over any $\gamma \in \Gamma^p$ is denoted by 
$\Pi_n(\gamma^*, \gamma)$, i.e., 
following \eqref{eq:post_prob}, we have
\begin{align*}
\Pi_n(\gamma^*, \gamma) := 
\frac{\pi(\gamma^* |  D_n)}{\pi(\gamma \;\, | D_n)} = 
\sfBF_n(\gamma^*, \gamma)
\times \frac{\pi(\gamma^*)}{\pi(\gamma)}.
\end{align*}
Thus, following Lemma \ref{lem:marg_gam} and the above definition,  we have
\begin{align}\label{eq:logBF}
\begin{split}
        \log \Pi_n(\gamma^*, \gamma) 
        &= -\frac{np}{2}(\log (V_n/n + R_n^*) - \log(V_n/n + R_n^\gamma))\\
        &\qquad \qquad \qquad - \frac{1}{2}(|\gamma^*| - |\gamma|)\log(1+g) + \log (\pi(\gamma^*)/\pi(\gamma)). 
        \end{split}
    \end{align}
    Furthermore, we have, by the weak law of large numbers, $V_n \to  \sum_{j = 1}^p \text{var}(X_j)$ in $\sfP^*$-probability and also, following \cite{rao1999linear} \S 2.3, for every $j \in [p]$, $R_{j, n}^\gamma \to r^\gamma_j$ and $R_{j, n}^* \to r^*_j$, again in $\sfP^*$-probability. 
    Now, suppose that $\gamma^* \not\subseteq \gamma$. Then 
    regarding the first term in the right hand side of \eqref{eq:logBF}, 
    \begin{align}\label{eq:def_delta}
       \log (V_n/n + R_n^*) - \log(V_n/n + R_n^\gamma)
        = - \, \delta_\gamma + o_p(1), 
    \end{align}
    where $\delta_\gamma := \log (r^\gamma/r^*) = \log r^\gamma - \log (p\sigma^2) > 0$ 
    since $r^\gamma > r^* = p\sigma^2$, following from Theorem \ref{thm:key_result}. 
    
Again, we have
    \begin{align*}
\pi(\gamma^* | D_n) &= \frac{m\lt(D_n \rt | \gamma^*) \pi_g(\gamma^*)}{\sum_{\gamma \in \Gamma^p}m\lt(D_n \rt | \gamma)  \pi_g(\gamma)} = \frac{1}{\sum_{\gamma \in \Gamma^p} \Pi_n(\gamma^*, \gamma)^{-1}} = \frac{1}{1 + \sum_{\gamma \neq \gamma^*} \Pi_n(\gamma^*, \gamma)^{-1}},
\end{align*}
which leads to
    \begin{align*}
        &1 - \pi(\gamma^* | D_n) = \frac{\sum_{\gamma \neq \gamma^*} \Pi_n(\gamma^*, \gamma)^{-1}}{1 + \sum_{\gamma \neq \gamma^*} \Pi_n(\gamma^*, \gamma)^{-1}}\\
        &\leq \sum_{\gamma \neq \gamma^*} \Pi_n(\gamma^*, \gamma)^{-1} = \sum_{\gamma^* \subset \gamma} \Pi_n(\gamma^*, \gamma)^{-1} + \sum_{\gamma^* \not\subseteq \gamma} \Pi_n(\gamma^*, \gamma)^{-1}\\
        &= \sum_{\gamma^* \subset \gamma} \exp(O_p(1))  (1 + g)^{(|\gamma^*| - |\gamma|)/2} 
        \frac{\pi(\gamma)}{\pi(\gamma^*)}
        + \sum_{\gamma^* \not\subseteq \gamma} \exp(-n(p\delta_\gamma/2 + o_p(1))) (1 + g)^{(|\gamma^*| - |\gamma|)/2} 
        \frac{\pi(\gamma)}{\pi(\gamma^*)}\\
        &\leq \exp(O_p(1))\sum_{\gamma^* \subset \gamma} (1 + n)^{(|\gamma^*| - |\gamma|)/2} + \sum_{\gamma^* \not\subseteq \gamma} \exp(-n(p\delta_\gamma/2 + o_p(1))) (1 + n)^{(|\gamma^*| - |\gamma|)/2}\\
        &\leq \exp(O_p(1)) \; n^{-1/2} + \exp(-(np/2)(\delta^* + o_p(1)))(1+n)^{|\gamma^*|/2}.
    \end{align*}
In the above, the third equality follows from \eqref{eq:logBF}, \eqref{eq:def_delta} and Lemma \ref{lem:mis_wilk_shift}. The last inequality follows from the fact that for every $\gamma \supset \gamma^*$, $|\gamma^*| - |\gamma| \leq -1$, the definition of $\delta^*$ and the fact that $|\gamma| \geq 0$. In particular, when $\gamma^*$ is a complete graph, there is no DAG $\gamma \supset \gamma^*$, and consequently, the first term in the right hand side does not appear in the calculation. The proof is complete. 
\end{proof}

\end{document}